\documentclass[12pt]{article}
\usepackage{amsmath}
\usepackage{amsthm}
\usepackage{amssymb}
\usepackage{graphics}
\newcommand{\be}{\begin{equation}}
\newcommand{\ee}{\end{equation}}
\newcommand{\ben}{\begin{eqnarray*}}
\newcommand{\een}{\end{eqnarray*}}
\newtheorem{examp}{\sc Example}
\newtheorem{remk}{\sc Remark}
\newtheorem{corol}{\sc Corollary}
\newtheorem{lemma}{\sc Lemma}
\newtheorem{theorem}{\sc Theorem}
\newtheorem{defn}{\sc Definition}
\newcommand{\bt}{\begin{theorem}}
\newcommand{\et}{\end{theorem}}
\newcommand{\bl}{\begin{lemma}}
\newcommand{\el}{\end{lemma}}
\newcommand{\bed}{\begin{defn}}
\newcommand{\eed}{\end{defn}}
\newcommand{\brem}{\begin{remk}}
\newcommand{\erem}{\end{remk}}
\newcommand{\bex}{\begin{examp}}
\newcommand{\eex}{\end{examp}}
\newcommand{\bcl}{\begin{corol}}
\newcommand{\ecl}{\end{corol}}

\newcommand{\NI}{\noindent}

\newcommand{\al}{\alpha}

\newcommand{\fal}{\forall}

\newcommand{\riro}{\Rightarrow}
\newcommand{\sbs}{\subset}

\newcommand{\seq}{\subseteq}
\newcommand{\dsp}{\displaystyle}
\newcommand{\vsp}{\vskip 1em}

\newcommand\Bigger[2][7]{\left#2\rule{0mm}{#1truemm}\right.}

\theoremstyle{definition}
\theoremstyle{remark}

\numberwithin{equation}{section}

\begin{document}
\title{On  Semimonotone  Star Matrices and Linear Complementarity Problem}
\author{R. Jana$^{a,1}$, A. K. Das$^{a, 2}$, S. Sinha$^{b, 3}$\\
	\emph{\small $^{a}$Indian Statistical Institute, 203 B. T.
		Road, Kolkata, 700 108, India.}\\
	\emph{\small $^{b}$Jadavpur University, Kolkata , 700 032, India.}\\	
	\emph{\small $^{1}$Email: rwitamjanaju@gmail.com}\\
	\emph{\small $^{2}$Email: akdas@isical.ac.in}\\
	\emph{\small $^{3}$Email: sagnik62@gmail.com }}
\date{}
\maketitle

\begin{abstract}
	
 \noindent In this article, we introduce the class of semimonotone star ($E_0^s$) matrices. We establish the importance of the class of $E_0^s$-matrices in the context of complementarity theory. We show that the principal pivot transform of $E_0^s$-matrix is not necessarily $E_0^s$ in general. However, we prove that $\tilde{E_0^s}$-matrices, a subclass of the $E_0^s$-matrices with some additional conditions, is in $E_0^f$ by showing this class is in $P_0.$ We prove that LCP$(q, A)$ can be processable by Lemke's algorithm if $A\in \tilde{E_0^s}\cap P_0.$ We find some conditions for which the solution set of LCP$(q, A)$ is bounded and stable under the $\tilde{E^s_0}$-property. We propose an algorithm based on an interior point method to solve LCP$(q, A)$ given $A \in \tilde{E^{s}_{0}}.$\\

\NI{\bf Keywords:} Linear complementarity problem, principal pivot transform, Lemke's algorithm, interior point method, semimonotone star matrix, $\tilde{E^s_0}$-matrix.\\ 

\NI{\bf AMS subject classifications:} 90C33, 90C90.
\end{abstract}
\footnotetext[1]{Corresponding author}


\section{Introduction}
\noindent The concept of pseudomonotone or copositive star matrices on a closed convex cone with respect to complementarity condition was studied by Gowda \cite{gowdastar}. The properties of copositive star matrices are well studied in the literature of linear complementarity problem. A star matrix \cite{Bazan} is defined as any point $x$ from solution set of LCP$(q, A)$ satisfies $A^T x\leq 0.$  Bazan and Lopez \cite{Bazan} studied $F_1$-matrix in the context of star matrices and proved the necessary and sufficient conditions of $F_1$-properties.
In linear complementarity theory, much of the research is devoted to find out constructive characterization of $Q_0$ and $Q$-matrices. The set $K(A)$ denotes the closed cone containing the nonnegative orthant $R^n_+.$ Eaves \cite{eaves} showed that $A\in Q_0$ if and only if $K(A)$ is convex. A subclass $Q$ of $Q_0$ is defined by the property that $A\in Q$ if and only if $K(A)=R^n.$ Aganagic and Cottle \cite{aganagic} showed that Lemke's algorithm processes LCP($q,A$) if $A\in P_0\cap Q_0.$

Many of the concepts and algorithms in optimization theory are developed based on principal pivot transform (PPT). The notion of PPT is originally motivated by the well known linear complementarity problem. The class of semimonotone matrices ($E_0$) introduced by Eaves \cite{eaves} (denoted by $L_1$ also) consists of all real square matrices $A$ such that LCP$(q, A)$ has a unique solution for every $q>0$. Cottle and Stone \cite{cottle1983uniqueness} introduced the notion of a fully semimonotone matrix $(E_{0}^{f})$ by requiring that  every PPT of such a matrix is a semimonotone matrix. Stone studied various properties of $E_{0}^{f}$-matrices and conjectured that $E_{0}^{f}$ with $Q_{0}$-property are contained in $P_{0}.$
The linear complementarity problem is a combination of linear and nonlinear system of inequalities and equations. The problem may be stated as follows:
Given $A\in R^{n\times n}$ and  a vector  $\,q\,\in\,R^{n},\,$ the {\it linear complementarity problem} LCP$(q,A)$ is the problem of finding a solution $w\;\in R^{n}\;$ and  $z\;\in R^{n}\;$ to the following system:
\be
\dsp {w\,-\,A z\;=\;q,\;\;w\geq 0,\;z\geq 0}
\ee
\be
\dsp {w^{T}\,z\,=\;0}
\ee

Let FEA$(q, A)$ $= \{z : q + Az \geq 0\}$ and SOL$(q, A)$ $= \{z \in$ {\normalfont FEA$(q, A) : z^T(q + Az) = 0\}$ denote the feasible and solution set of LCP$(q, A)$ respectively. In this article, we introduce a class of matrices called semimonotone star matrix $(E_0^s)$ by introducing the notion of star property. 
\vsp

The outline of the article is as follows. In Section 2, some notations, definitions and results are presented  that are used in the next sections. In section 3, we introduce semimonotone star $(E_0^s)$-matrix and study some properties of this class in connection with complementarity theory, principal pivot transform. Section 4 deals with PPT based matrix classes under $E_0^s$-property. In section 5, we consider the SOL$(q, A)$ under $E_0^s$-property. In this connection we partially settle an open problem raised by Gowda and Jones \cite{jones-gowda}. We propose an iterative algorithm \cite{das2016generalized} to process LCP$(q, A)$ where $A \in \tilde{E^{s}_{0}},$ a subclass of $E^s_0$-matrix in section 6. A numerical example is presented to demonstrate the performance of the proposed algorithm in section 7.

\section{Preliminaries}
\noindent We denote the $n$ dimensional real space by $R^n$ where $R^{n}_{+}$ and $R^{n}_{++}$ denote the nonnegative and positive orthant of $R^{n}$ respectively. We consider vectors and matrices with real entries. For any set $\,\beta \seq \{1,2,\ldots,n\},\,$ $\bar{\beta}$ denotes its complement in $\{1,2,\ldots,n\}.$ Any vector $x\in R^{n}$ is a column vector unless otherwise specified. For any matrix $A\,\in\,R^{n\times n},$ $a_{ij}$ denotes its $i$th row  and $j$th column entry, $A_{\cdot j}$ denotes the $j$th column and $A_{i\cdot}$ denotes the $i$th row of $A$. If $A$ is a matrix of order $n$,  $\emptyset\neq\al \seq \{1,2,\ldots,n\}$ and $\emptyset\neq\beta\seq \{1,2,\ldots,n\},$ then  $A_{\al\beta}$ denotes the submatrix of $A$ consisting of only the rows and columns of $A$ whose indices are in $\al$ and $\beta,$ respectively. For any set $\al$, $\;|\al|$ denotes its cardinality. $\|A\|$ and $\|q\|$ denote the norms of a matrix $A$ and a vector $q$ respectively. 

The {\it principal pivot transform} (PPT) of $A,$ a real $n \times n$ matrix, with respect to $\al\subseteq
\{1,2,\ldots,n\}$ is defined as the matrix given by
$$ \dsp {
	M = \left[ \begin{array}{cc}
	M_{\al\al} & M_{\al\bar{\al}}\\
	M_{\bar{\al}\al} & M_{\bar{\al}\bar{\al}}
	\end{array} \right]
} $$
\NI{where}  $M_{\al\al} =(A_{\al\al})^{-1},\;
M_{\al\bar{\al}}$=$-(A_{\al\al})^{-1}
A_{\al\bar{\al}},\;\,$$M_{\bar{\al}\al}=
A_{\bar{\al}\al}(A_{\al\al})^{-1},$ $M_{\bar{\al}\bar{\al}}=
A_{\bar{\al}\bar{\al}}-A_{\bar{\al}\al}(A_{\al\al})^{-1}
A_{\al\bar{\al}}.$ Note that PPT is only defined with respect to those $\al$ for which $\det A_{\al\al} \neq 0.$ By a {\it legitimate principal pivot transform} we mean the PPT obtained from $A$ by performing a principal pivot on its nonsingular principal submatrices. When $\al=\emptyset$, by convention $\det A_{\al\al}=1$ and $M=A.$  For further details see \cite{cottleppt}, \cite{cps}, \cite{neogy2005principal} and \cite{neogy2012generalized} in this connection. The PPT of LCP$(q,A)$ with respect to $\al$ (obtained by
pivoting on $A_{\al\al}$) is given by
LCP$(q^{'},M)$ where $M$ has the same structure already mentioned with $q^{'}_{\al}= -A_{\al\al}^{-1}q_{\al}$ and
$q^{'}_{\bar{\al}} =q_{\bar{\al}}-A_{\bar{\al}\al}A_{\al\al}^{-1}q_{\al}.$

\vsp
We say that $A\in R^{n\times n}$ is \\

	\NI $-$ {\it positive definite} (PD) matrix if $x^{T}Ax> 0,\;\fal\;0\neq x\in R^{n}.$ \\ 
	\NI $-$ {\it positive semidefinite} (PSD) matrix if $x^{T}Ax\geq 0,\;\fal\;x\in R^{n}.$ \\
	\NI $-$ {\it column sufficient} matrix if $x_i(Ax)_i \leq 0 \ \forall i$ $\implies \ x_i(Ax)_i = 0 \ \forall i.$\\
	\NI $-$ {\it row sufficient} matrix if $A^T$ is column sufficient.\\
	\NI $-$ {\it sufficient} matrix if $A$ is both column and row sufficient.\\ 
	\NI $-$ {\it $P\,(P_{0})$}-matrix if all its principal minors are  positive (nonnegative).\\
	\NI $-$ {\it $N(N_{0})$}-matrix  if  all its principal minors are  negative (nonpositive).\\
	\NI $-$ {\it copositive} $(C_{0})$ matrix if $x^{T}Ax\geq 0,\;\fal\;x\geq 0.$ \\
	\NI $-$ {\it strictly copositive} $(C)$ matrix if $x^{T}Ax> 0,\;\fal\;0\neq\;x\geq 0.$ \\
	\NI $-$ {\it copositive plus} $(C_{0}^{+})$ matrix if $A$ is copositive and $x^TAx=0,\; x\geq 0 \implies (A + A^T)x = 0.$\\
	\NI $-$ {\it copositive star} $(C_{0}^{*})$ matrix if $A$ is copositive and $x^TAx=0,\;Ax\geq 0,\;x\geq 0 \implies A^Tx\leq 0.$\\
	\NI $-$ {\it semimonotone}  ($E_{0}$) matrix if for every $0\neq x\geq 0,$  $\exists$  an $i$ such that $x_{i}\,>\,0$ and $(Ax)_{i}\geq 0.$ \\
	\NI $-$ $L_2$-matrix if for every $0 \neq x \geq 0, \ x \in R^n,$ such that $Ax \geq 0, \ x^TAx = 0,$ $\exists$ two diagonal matrices $D_1 \geq 0$ and $D_2 \geq 0$ such that $D_2x \neq 0$ and $(D_1A + A^TD_2)x = 0.$\\
	\NI $-$ $L$-matrix if it is $E_0 \cap L_2.$\\ 
	\NI $-$ {\it strictly semimonotone}  ($E$) matrix if for every $0\neq x \geq 0,$  $\exists$  an $i$ such that $x_{i}\,>\,0$ and $(Ax)_{i}> 0.$ \\
	\NI $-$ {\it pseudomonotone} matrix if for all $x, y \geq 0,$ $(y-x)^{T}Ax \geq 0 \implies (y-x)^{T}Ay \geq 0.$ \\
	\NI $-$ {\it positive subdefinite matrix} (PSBD) if $\forall x \in R^n,$ $x^TAx < 0 \implies$ either $A^Tx \leq 0$ or $A^Tx \geq 0.$\\
	\NI $-$ {\it fully copositive} $(C_{0}^{f})$ matrix if every legitimate PPT  of $A$ is $C_{0}$-matrix.\\
	\NI $-$ {\it fully semimonotone} $(E_{0}^{f})$ matrix if every legitimate PPT  of $A$ is $E_{0}$-matrix. \\
	\NI $-$ {\it almost $P_{0}(P)$}-matrix if  $\det A_{\al\al}\geq 0$ $(>0)$ 
	$\fal\;\al\sbs\{1,2,\ldots,n\}$ and $\det A<0.$ \\
	\NI $-$ {\it an almost $N_{0}(N)$}-matrix if  $\det A_{\al\al}\leq 0$ $(<0)$ $\fal\;\al\sbs\{1,2,\ldots,n\}$ and $\det A>0.$ \\
	\NI $-$ {\it almost copositive} matrix if it is copositive of order $n-1$ but not of order $n.$ \\
	\NI $-$ {\it almost E} matrix if it is E of order $n-1$ but not of order $n.$\\
	\NI $-$ {\it almost fully copositive} ({\it almost} $C_{0}^{f}$) matrix if its PPTs are either $C_{0}$ or almost $C_{0}$ and  there exists atleast one PPT $M$ of  $A$ for some $\al\sbs\{1,2,\ldots,n\}$ that is almost $C_{0}.$ \\
	\NI $-$ {\it copositive of exact order} $k$ matrix if it is copositive  up to order
	$n-k.$\\
	\NI $-$ $Z$-matrix if $a_{ij} \leq 0.$ \\
	\NI $-$ $K_0$-matrix \cite{chu} if it is $Z$-matrix as well as $P_0$-matrix. \\
	\NI $-$ {\it connected} $(E_c)$ matrix if $\forall q,$ LCP$(q,A)$ has a connected solution set.\\
	\NI $-$ $R$-matrix if $\nexists$ $z \in R^n_{+}, \ t (\geq 0) \in R$ satisfying 
	\begin{center}
		$A_{i.}z + t = 0$ if $i$ such that $z_i > 0,$\\
		$A_{i.}z + t \geq 0$ if $i$ such that $z_i = 0.$
	\end{center}
	\NI $-$ $R_0$-matrix if LCP$(0, A)$ has unique solution.\\
	\NI $-$ $Q_b$-matrix if SOL$(q, A)$ is nonempty and compact $\forall q \in R^n.$\\
	\NI $-$ {\it $Q$}-matrix if for every $q\in R^{n},$ LCP$(q,A)$ has a solution. \\
	\NI $-$ {\it $Q_{0}$}-matrix if for any $q\in R^{n},$ (1.1) has a solution implies that LCP$(q, A)$ has a solution. \\
	\NI $-$ {\it completely $Q$}-matrix $(\bar{Q})$ if  all its principal submatrices are $Q$-matrices. \\
	\NI $-$ {\it completely $Q_{0}$}-matrix $(\bar{Q_{0}})$ if  all its principal submatrices are $Q_{0}$-matrices. \\
	
We state some game theoretic results due to von Neumann \cite{von} which are needed in the sequel. In a two person zero-sum matrix game, let $v(A)$ denote the value of the game corresponding to the pay-off matrix $A.$
The value of the game $v(A)$ is {\it positive} ({\it nonnegative}) if  there exists a $0\neq x\geq 0$ such that $Ax>0\;(Ax\geq 0).$ Similarly, $v(A)$ is {\it negative} ({\it nonpositive}) if  there exists a
$0\neq y\geq 0$ such that $A^{T}y<0\;(A^{T}y\leq 0).$
\vsp
The following result was proved by V$\ddot{\mbox{a}}$liaho \cite{valiaho} for symmetric almost copositive  matrices. However this is true for nonsymmetric almost copositive matrices as well.
\bt \cite{akd} \label{das}
Let $A\in R^{n\times n}$ be  almost copositive. Then
$A$ is PSD of order $n-1,$ and $A$ is PD of order $n-2.$
\et
\bt \cite{mohan2001more} \label{psbd}
Suppose $A \in R^{n \times n}$ is a PSBD matrix and rank$(A) \geq 2.$ Then $A^T$ is PSBD and at least one of the following conditions hold:\\
(i) $A$ is a PSD matrix.\\
(ii) $(A + A^T) \leq 0.$\\
(iii) $A \in C_0^*.$
\et 
\bt \cite{mohan2001more} \label{SS}
Suppose $A \in R^{n \times n}$ is a PSBD matrix and rank$(A) \geq 2.$ and $A + A^T \leq 0.$ If A is not a skew-symmetric matrix, then $A \leq 0.$
\et
Here we consider some more results which will be required in the next section.
\bt \cite{chu} \label{chu}
Suppose $A \in R^{n \times n}$ with $A$ satisfies $(++)$-property. If $A \in E_0$ then $A \in P_0.$
\et
\bt \cite{gowda-pang} \label{pangresult}
Let $A \in R^{n \times n}$ be given. Consider the statements\\
(i) $A \in R.$\\
(ii) $A \in$ int$(Q) \cap R_0.$\\
(iii) the zero vector is a stable solution of the LCP$(0, A).$\\
(iv) $A \in Q \cap R_0.$\\
(v) $A \in R_0.$\\
Then the following implications hold:
(i) $\implies$ (ii) $\implies$ (iii) $\implies$ (iv) $\implies$ (v).
Moreover, if $A \in E_0,$ then all five statements are equivalent.
\et
\bt \cite{gowda-pang} \label{pangtheorem}
Let $A \in$ int$(Q) \cap R_0.$ If the LCP$(q, A)$ has a unique solution $x^{*},$ then LCP$(q, A)$ is stable at $x^{*}.$
\et

\bt \cite{tpsriparna} \label{tp}
Let $A \in R^{n \times n}$ be such that for some index set $\al$ (possibly empty), $A_{\bar{\al}} = 0.$ If $A_{\al \al} \in P_{0} \cap Q,$ then SOL$(q, A)$ is connected for every q.
\et

\bt \cite{cao-ferris} \label{cao}
Suppose $A \in E_c \cap Q_0.$ Then  Lemke's algorithm terminates at a solution of LCP$(q, A)$ or determines that FEA$(q, A) = \emptyset.$
\et
\begin{theorem} \cite{gowdastar} \label{gowdatheorem}
Suppose that $A$ is pseudomonotone on $R^{n}_{+}.$ Then $A$ is a $P_0$ matrix.
\end{theorem}

\bt \cite{jones-gowda} \label{fullysemi}
Suppose that $A \in R^{n \times n} \cap E_c.$ Then $A \in E^{f}_{0}.$
\et


\bt \cite{valiaho} \label{valiaho_res}
Any $2 \times 2$ $P_0$-matrix with positive diagonal is sufficient.
\et

\bt \cite{eaves} \label{L}
$L$-matrices are $Q_0$-matrices.
\et

\bt \cite{cottle-guu} \label{guu}
Let $A \in R^{n \times n}$ where $n \geq 2.$ Then $A$ is sufficient if and only if $A$ and each of its principal pivot transforms are sufficient of order $2.$
\et

\bt \cite{Pang} \label{Qmatrix} \cite{neogy2005almost}
Suppose $A \in E_0 \cap R^{n \times n}.$ If $A \in R_0$ then $A \in Q.$
\et

\bt \cite{Bazan} \label{lopez}
$Q_b = Q \cap R_0.$
\et

\section{Some properties of $E_0^s$-matrices}
We begin by the definition of semimonotone star $(E_0^s)$ matrix.
\bed
A semimonotone matrix $A$ is said to be a semimonotone star ($E_0^s$) matrix if $x^TAx=0,\;Ax\geq 0,\;x\geq 0 \implies A^Tx\leq 0.$
\eed
\vsp
\begin{examp}
Consider the matrix
	$ \dsp {
		A = \left[ \begin{array}{rr}
		0 & -5\\
		2 & 0
		\end{array} \right]
	} .$
	Now $x^TAx = -3x_1x_2.$ Consider $\dsp{
		x = \left[ \begin{array}{r}
		k_1\\
		k_2
		\end{array} \right]
	},$ where $k_1, \ k_2 \geq 0.$ Hence we consider the following cases.\\
		
		\NI Case I: For $k_1 = k_2 = 0,$ $x = 0, \ Ax = 0, \ x^TAx = 0$ implies $A^Tx = 0.$\\
		\NI Case II: For $k_1 > 0 , \ k_2 = 0,$ $x \geq 0, \ Ax \geq 0, \ x^TAx = 0$ implies $A^Tx \leq 0.$ \\
		\NI Case III: For $k_1 = 0, \ k_2 > 0,$ $x \geq 0.$ However $Ax \ngeq 0.$\\
\NI Hence $A \in E_{0}^{s}.$
\end{examp}

The following result shows that $E_0^s$-matrices are invariant under principal rearrangement and scaling operations.

\bt
If $A\in R^{n\times n}\cap E_0^s$-matrix and $P\in R^{n\times n}$ is any permutation matrix, then $PAP^{T}\in$  $E_0^s.$
\et
\begin{proof}
Let $A\in E_0^s$ and let $P\in R^{n\times n}$ be
any permutation matrix. Then $PAP^T$ is an $E_0$-matrix by the Theorem 4.3 of \cite{tsatsomeros2019semimonotone}. Now for any $x\in R^{n}_{+},$ let $y=P^{T}x.$
Note that $x^{T}PAP^{T}x=y^{T}Ay =0,$ $AP^Tx=Ay\geq 0 $ $\riro$
$A^{T}y=A^{T}P^{T}x\leq 0,$
since $P$ is a permutation matrix. It follows that
$PAP^{T}$ is a $E_0^s$-matrix. The converse  of the above theorem follows from the fact that $P^{T}P=I$
and $A=P^{T}(PAP^{T})(P^{T})^{T}.$
\end{proof}
\begin{examp}
	Let
	$ \dsp{
		A = \left[ \begin{array}{rrr}
		0 & 1 & 1\\
		2 & 0 & 2\\
		-4 & -5 & 0
		\end{array} \right].
	}$ Clearly, $A\in $ $E_0.$ The nonzero vectors in SOL$(0, A)$ are of the form $ \dsp{
	x = \left[ \begin{array}{rrr}
	0 \\
	0 \\
	k
	\end{array} \right]
}$ for $k > 0,$ and for such $x$ the inequality $A^{T}x \leq
0$ holds. Therefore $A \in E_0^s.$ Consider $ \dsp{
	P = \left[ \begin{array}{rrr}
	0 & 0 & 1\\
	1 & 0 & 0\\
	0 & 1 & 0
	\end{array} \right].
}$ We get\\ $ \dsp{
PAP^T = \left[ \begin{array}{rrr}
0 & -4 & -5\\
1 & 0 & 1\\
2 & 2 & 0
\end{array} \right].
}$ Hence $x^{T}PAP^Tx=0,$  $ \dsp {
PAP^Tx\geq 0,\; x\geq 0 } $ imply $PA^{T}P^Tx \leq
0.$ Therefore $PAP^T$ is a $E^{s}_{0}$-matrix.
\end{examp}
\bt
Suppose $A\in R^{n\times n}$ is a $E_0^s$-matrix. Let $D\in R^{n\times n}$ be a positive diagonal matrix. Then  $A\in$  $E_0^s$ if and only if $DAD^{T}\in$ $E_0^s.$
\et
\begin{proof}Let $A\in E_0^s.$  For any $x\in R^{n}_+,$ let $y=D^{T}x.$
Note that $x^{T}DAD^{T}x=y^{T}Ay=0,$ $AD^Tx=Ay\geq 0$ $\riro$
$A^{T}y=A^{T}D^{T}x\leq 0$
since $D$ is  a positive diagonal matrix. Thus $DAD^{T}\in E_0^s.$ The converse follows from the fact that $D^{-1}$ is a positive
diagonal matrix and $A=D^{-1}(DAD^{T})(D^{-1})^{T}.$
\end{proof}

The following example shows that  $A\in E_0^s$-matrix does not imply $(A+A^{T})\in E_0^s$-matrix.

\begin{examp}
Let
	$ \dsp {
		A = \left[ \begin{array}{rrr}
		0 & 1 & 1\\
		2 & 0 & 1\\
		-1 & -1 & 0
		\end{array} \right].
	} $ Clearly $A\in $ $E_0^s,$ since $x^{T}Ax=0,$  $ \dsp {
		Ax\geq 0,\; x\geq 0 } $ imply  $A^{T}x\leq
	0.$
	
	\NI It is easy to show that $ \dsp {
		A+A^{T} = \left[ \begin{array}{rrr}
		0 & 3 & 0\\
		3 & 0 & 0\\
		0 & 0 & 0
		\end{array} \right]  } $ is not a $E_0^s$-matrix.
\end{examp}

We show that PPT of $E_0^s$-matrix need not be $E_0^s$-matrix.
\begin{examp}
Consider the matrix
	$ \dsp {
		A = \left[ \begin{array}{rrr}
		0 & 1 & 1\\
		2 & 0 & 1\\
		-1 & -1 & 0
		\end{array} \right]
	}.$ Note that $A\in E_0^{s}$
	and it is easy to show that $ \dsp {
		A^{-1} = \frac{1}{3} \left[ \begin{array}{rrr}
		-1 & 1 & -1\\
		1 & -1 & -2\\
		2 & 1 & 2
		\end{array} \right]  } $ is not a $E_0^s$-matrix. Therefore any PPT of  $E_0^s$-matrix need  not be  $E_0^s$-matrix.
\end{examp}

Note that a matrix is in $E_0$ if and only if its transpose is in $E_0.$ We show that $A\in E_0^s$-matrix does not imply $A^{T}\in E_0^s$-matrix in general.

\begin{examp} 
	Consider the matrix
	$ \dsp {
		A = \left[ \begin{array}{rrr}
		0 & 1 & 1\\
		2 & 0 & 1\\
		-1 & -1 & 0
		\end{array} \right]
	} .$ Note that $A\in E_0^s.$
	However $ \dsp {
		A^{T} = \left[ \begin{array}{rrr}
		0 & 2 & -1\\
		1 & 0 & -1\\
		1 & 1 & 0
		\end{array} \right]} $ is not a $E_0^s$-matrix.
\end{examp}

Now we show a condition under which $A^T$ satisfies $E^s_0$-property.

\begin{theorem}
	Suppose that $A$ is pseudomonotone on $R^{n}_{+}$ and $A^T \in R_0.$ Then $A^T$ satisfies $E^s_0$-property.
\end{theorem}
\begin{proof}
	Since $A$ is pseudomonotone on $R^{n}_{+},$ then $A$ is $P_0$ matrix by Theorem \ref{gowdatheorem}. Hence $A \in E_0.$ We have to show that $A^T$ satisfies the following property.
	\begin{center}
		$0 \neq x \geq 0,$ \ $A^Tx \geq 0,$ and $x^TA^Tx = 0$ $\implies$ $Ax \leq 0.$
	\end{center}
	
Since $A^T \in R_0$ then $0 \neq x \geq 0,$ $A^Tx = 0$ has no solution. Therefore for atleast one $i,$ $(A^Tx)_i > 0.$ Let us consider the vector $e_i$ which has one at the $i$th position and zeros elsewhere. Now consider $y = e_i + \lambda e_j,$ where $i \neq j$ and $\lambda \geq 0.$ Then, for any small $\delta > 0,$ we get
	\begin{center}
		$(x - \delta y)^T A (\delta y) = \delta[(A^Tx)_i + \lambda (A^Tx)_j - \delta y^TAy] \geq 0.$
	\end{center}
	By pseudomonotonicity, $(x - \delta y)^T A x \geq 0.$ Thus $y^TAx \leq 0.$ This gives $(Ax)_i + \lambda (Ax)_j \leq 0.$ As $\delta$ is arbitrary, $(Ax)_i \leq 0$ and $(Ax)_j \leq 0.$ Hence $Ax \leq 0.$
\end{proof}
	
	
\begin{corol}
		Suppose that $A$ is pseudomonotone on $R^{n}_{+},$ and satisfies one of the following conditions:
		\begin{description}
			\item[(i)] $A$ is invertible.
			\item[(ii)]$A$ is normal i.e. $AA^T = A^TA.$
		\end{description}
		Then $A^T \in E^s_0.$
\end{corol}
Note that $C_{0}^{*} \subseteq E_0^s.$
\bed
A matrix $A$ is said to be completely semimonotone star ($\bar E_0^s$) matrix if all its principal submatrices are semimonotone star matrix.
\eed

We say that $E_{0}^{s}$ is not a complete class which can be illustrated with the following example.

\begin{examp} Consider the matrix
	$ \dsp {
		A = \left[ \begin{array}{rrr}
		0 & 1 & 1\\
		2 & 0 & 1\\
		-4 & -5 & 0
		\end{array} \right]
	} .$  Note that $A\in E_0^s.$
	It is easy to show that $ \dsp {
		A_{12} = \left[ \begin{array}{rr}
		0 & 1 \\
		2 & 0
		\end{array} \right]  } $ is not a $E_0^s$-matrix.
\end{examp}
 \begin{theorem} \label{psbd res}
 	Let $A \in$ PSBD \,$\cap \ E_0$ with rank$(A) \geq 2.$ Further suppose $A$ is not a skew symmetric matrix. Then $A \in E^s_0$-matrix.
 \end{theorem}
\begin{proof}
	Let $A$ be PSBD matrix with rank$(A) \geq 2.$ By the Theorem \ref{psbd} we have following three cases.\\
		Case I: $A$ is PSD matrix. This implies $A \in E^s_0.$\\
		Case II: $A \in C_0^*.$ This implies  $A \in E^s_0.$\\
		Case III: $(A + A^T) \leq 0.$ As $A$ is a PSBD matrix with rank$(A) \geq 2,$ $A \leq 0$ by the Theorem \ref{SS}. Note that $A \in E_0.$ To show $A \in E^s_0,$ consider $x^TAx=0,\;Ax\geq 0,\;x\geq 0$ which implies $A^Tx \leq 0.$ Hence $A \in E^s_0.$
	\end{proof}
\begin{examp}
	 Consider the matrix
		$ \dsp {
			A = \left[ \begin{array}{rr}
			0 & 3\\
			-1 & 0
			\end{array} \right]
		} .$ It is easy to show that $A$ is PSBD matrix with rank$(A) \geq 2.$ Hence by Theorem \ref{psbd res}, $A \in E_0^s.$
\end{examp}

\section{PPT based matrix classes under $E_0^s$-property}
We consider some PPT based matrix classes with $E^{s}_{0}$-property in the context of linear complementarity problem to show that these classes are processable by Lemke's algorithm under certain condition. We settle the processability of Lemke's algorithm through identification of a new subclass of $P_0\cap Q_0$-matrices.
\begin{defn}
	A matrix $A \in E^{s}_{0}$ is said to be $\tilde{E^s_0}$-matrix if for $x \in$ SOL$(0, A),$ $(A^Tx)_i \neq 0$ $\implies$ $(Ax)_i \neq 0$ $\forall$ $i.$ 
\end{defn}
\begin{examp}
	Consider	$ \dsp{
		A = \left[ \begin{array}{rrr}
		0 & 1 & 1\\
		2 & 0 & 2\\
		-2 & -4 & 0
		\end{array} \right].
	}$
Note that, $A \notin C^{*}_0.$ For $k > 0$ and 
	$ \dsp{
	x = \left[ \begin{array}{r}
	0 \\
	0 \\
	k
	\end{array} \right],
}$ 
$x \geq 0, \ Ax \geq 0, \ x^TAx = 0$ implies $A^Tx \leq 0.$ Hence $A \in E^s_0.$ Now $ \dsp{
	A^Tx = \left[ \begin{array}{r}
	-2 \\
	-4\\
	0
	\end{array} \right]
}$ and $ \dsp{
Ax = \left[ \begin{array}{r}
1 \\
2 \\
0
\end{array} \right]
}.$ Therefore $\forall$ $i,$ $(A^Tx)_i \neq 0$ $\implies$ $(Ax)_i \neq 0.$ Hence $A \in \tilde{E^s_0}.$ 
\end{examp}
\begin{remk}
	It is easy to show that $C_0^+ \subseteq \tilde{E^s_0}.$
\end{remk}
Note that not every $E^s_0$-matrix is $\tilde{E^s_0}$-matrix. We consider the following example from the paper \cite{jana2018}.

\begin{examp}
	Consider $ \dsp{
		A = \left[ \begin{array}{rrr}
		1 & -1 & -2\\
		-1 & 1 & 0\\
		0 & 0 & 1
		\end{array} \right].
	}$ Note that $A \in P_0.$ Hence $A \in E_0.$ The only nonzero vectors in SOL$(0, A)$ are of the form  
$ \dsp{
x = \left[ \begin{array}{r}
k \\
k \\
0
\end{array} \right]
}$ 
for $k > 0.$ Now for such $x,$ $A^Tx \leq 0$ holds. Hence $A \in E^s_0.$ Now $ \dsp{
A^Tx = \left[ \begin{array}{r}
0 \\
0\\
-2k
\end{array} \right]
}$ and $ \dsp{
Ax = \left[ \begin{array}{r}
0 \\
0 \\
0
\end{array} \right]
}.$ Note that $(A^Tx)_3 \neq 0$ but $(Ax)_3 = 0.$ Hence $A \notin \tilde{E^s_0}.$ 
\end{examp}

\bt \label{L_2}
Let $A\in R^{n\times n}\cap \tilde{E_{0}^s}.$ Assume that every legitimate PPT of $A$ is either almost  $E$ or completely $\tilde{E_{0}^s}.$ Then $A\in P_{0}.$
\et

\begin{proof} Since $A \in \tilde{E_0^s}$, we say $x \in$ SOL$(0, A)$ which implies $A^Tx \leq 0.$ Again by definition $(A^Tx)_i \neq 0$ which implies $(Ax)_i \neq 0$ $\forall$ $i.$ Now by taking $D_{2} = I,$ where $I$ represents the identity matrix. Then $D_{2}x = Ix \neq 0.$ So $(D_{1}A + A^TI)x = 0$ by taking,
\begin{center}
\begin{displaymath}
	D_{ii}=
	\Bigger[14]\{\begin{array}{@{}cl}
		\dfrac{-(A^Tx)_{i}}{(Ax)_{i}}, & (Ax)_{i} \neq 0,\\[3mm]
		0,              & (Ax)_{i} = 0 .
	\end{array}
\end{displaymath}

\end{center}
Where $D_{ii}$ denotes the $i$th diagonal of $D_{1}.$ So $A \in E_{0}^s\cap L_{2}.$ Therefore $A \in Q_0$ by the Theorem \ref{L}. For rest of the proof, we follow the approach given by Das \cite{akd}. However for the sake of completeness we give the proof.
Note that every legitimate PPT of $A$ is either almost  $E$ or completely $\tilde{E_{0}^s}.$ Suppose $M$ is a PPT of $A$ so that  $M\in$ almost $E.$ Then all principal submatrices of $M$ upto $n-1$ order are $\bar{Q}.$ Hence $M \in \bar{Q_0}.$ Since the PPT of $A$ is either almost $E$ or completely $\tilde{E_{0}^s},$ it follows that all proper principal submatrices are $P_0.$

Now to complete the proof, we need to show that $\det A \geq 0.$
Suppose not. Then det $A <0.$ This implies that $A$
is an almost $P_{0}$-matrix. Therefore $A^{-1} \in N_{0}.$ If $A^{-1}\in$ almost $E$ then this contradicts that the diagonal entries are positive.
Therefore det $A\geq 0.$ It follows that $A\in P_{0}.$

\end{proof}

\bcl
Let $A\in R^{n\times n}\cap \tilde{E_{0}^s}.$ Assume that every legitimate PPT of $A$ is either almost $E$ or completely $\tilde{E_{0}^s}.$ Then LCP$(q,A)$ is processable by Lemke's algorithm.
\ecl
 Earlier Das \cite{akd} proposed \textit{exact order 2 $C_0^f$}-matrices in connection with PPT based matrix classes. We define \textit{exact order k $C_0^f$}-matrices.
 \begin{defn}
 	$A$ is said to be an \textit{exact order k $C_0^f$}-matrix if its PPTs are either exact order \textit{k} $C_0$ or $E_0$ and there exists at least one PPT $M$ of $A$ for some
 	$\alpha \subset \{1, 2, \cdots, n\}$ that is exact order \textit{k} $C_0.$
 \end{defn}
 We prove the following theorem.
\begin{theorem}
	Let $A \in \tilde{E^{s}_{0}} \ \cap$ exact order $k$ $C_0^f$ $(n \geq k+2).$ Assume that PPT of $A$ has either exact order $k$ $C_0$ or $E_0$ with at least $k$ positive diagonal entries. Then LCP$(q, A)$ is processable by Lemke's algorithm.
\end{theorem}
\begin{proof}
	We show that $A\in P_0.$ Suppose $M$ is a PPT of $A$ so that  $M\in$ exact order \textit{k} $C_{0}.$ By Theorem \ref{das}, all
	the principal submatrices of order $(n-k)$ of $M$ are PSD. Now to show $M^{(n-k+1)} \in P_{0}$
	it is enough to show that det $M^{(n-k+1)}\geq 0.$ Suppose not. Then $\det M^{(n-k+1)} < 0.$ We consider $B=M^{(n-k+1)}$
	is an almost $P_{0}$-matrix. Therefore $B^{-1}\in N_{0}$ and there exists a nonempty subset $\al\sbs \{1,2,\ldots,n-1\}$ satisfying \cite{akd}
	\be \label{equation}
	B^{-1}_{\al\al}\leq 0,\;B^{-1}_{\bar{\al}\bar{\al}}\leq 0,\;B^{-1}_{\al\bar{\al}}\geq
	0\mbox{ and }B^{-1}_{\bar{\al}\al}\geq 0.
	\ee
	By definition $B^{-1}\in$ $E_{0}$ with atleast $k$ positive diagonal entry. This contradicts Equation \ref{equation}. Therefore $\det M^{(n-k+1)}\geq 0.$
	Now by the same argument as above, we show that  det $M\geq 0$. Therefore it
	follows that of $A \in P_{0}.$ Hence $A \in P_0 \cap \tilde{E^{s}_{0}}.$ So LCP$(q, A)$ is processable by Lemke's Algorithm.
\end{proof}
	
A matrix $A$ is said to satisfy $(++)$-property if there exists a matrix $X \in K_0$ such that $AX$ is a $Z$-matrix. For details see \cite{chu}. Now we establish the condition under which a matrix $A$ is sufficient satisfying $(++)$-property.
\bt
Suppose $A \in R^{n \times n} \cap E_0$ satisfies $(++)$-property. If each legitimate PPT of $A$ are either almost $C_0$ or completely $\tilde{E_{0}^{s}}$ with full rank second order principal submatrices, then $A$ is sufficient.
\et
\begin{proof}
	As $A \in E_0$ with $(++)$-property. Hence $A \in P_0$ by Theorem \ref{chu}. Suppose $M$ be a PPT of $A.$ We consider the following cases.\\	
		
		\NI Case I: If $M$ be almost $C_0,$ then by the Theorem \ref{das}, $M$ is PSD of order $(n-1).$ Hence $M$ is PSD of order 2 also. So by the Theorem \ref{guu}, $M$ is sufficient of order $(n-1).$ \\
		\NI Case II: If $M$ is completely $\tilde{E_{0}^{s}}$ then sign pattern of all $2 \times 2$ submatrices of $M$ will be in the following subcases: \\
		\NI Subcase I: If the sign pattern is $ \dsp {
			\left[ \begin{array}{rr}
			0 & +\\
			- & 0
			\end{array} \right]
		}$ or $ \dsp {
			\left[ \begin{array}{rr}
			0 & -\\
			+ & 0
			\end{array} \right]
		} $ then these two patterns are sufficient. \\
		\NI Subcase II: If the sign pattern is $ \dsp {
			\left[ \begin{array}{rr}
			+ & +\\
			- & +
			\end{array} \right]
		}$ or $ \dsp {
			\left[ \begin{array}{rr}
			+ & -\\
			+ & +
			\end{array} \right]
		} $ then by the Theorem \ref{valiaho_res} these two patterns are sufficient.\\
		\NI Subcase III: If the sign pattern is $ \dsp {
			\left[ \begin{array}{rr}
			+ & +\\
			+ & +
			\end{array} \right]
		}$
		then by the Theorem \ref{valiaho_res} this pattern is sufficient.\\
		\NI Subcase IV: If the sign pattern is $ \dsp {
			\left[ \begin{array}{rr}
			+ & -\\
			- & +
			\end{array} \right]
		}$
		then by the Theorem \ref{valiaho_res} this pattern is sufficient.\\
\NI Then for every PPT, $A_{\al \al}$ of order $2$ are sufficient. By the Theorem \ref{guu}, $A$ is sufficient.
\end{proof}

\section{Properties of SOL$(q, A)$ under $E_{0}^{s}$-property}
We show that solution set of LCP$(q, A)$ is connected if $A \in E^{s}_{0}$ with the following structure
$A$ =	$\begin{bmatrix}
A_{\al\al} & + \\
- & 0 \\
\end{bmatrix},$ where $A_{\alpha \alpha} \in R^{(n-1) \times (n-1)}.$
\bt \label{structure}
Let $A \in R^{n \times n}$ with $A$ =	$\begin{bmatrix}
A_{\al\al} & + \\
- & 0 \\
\end{bmatrix}$ and $A_{\al \al} \in P_0.$ Then $A \in \tilde{E_{0}^{s}}$-matrix.
\et
\begin{proof} First we show that $A$ =	$\begin{bmatrix}
	A_{\al\al} & + \\
	- & 0 \\
	\end{bmatrix}$ with $A_{\al\al} \in P_{0}$ is $E_0$-matrix. Let us consider $(u_\al, v) \in R^{n}_{+}$ be a given vector where $\al = \{1, 2, \cdots, (n-1)\}.$ Without loss of generality we assume $u_\al \neq 0.$ Now as $A_{\al \al} \in P_0,$ we can write $A_{\al \al} \in E_0.$ By semimonotonicity $A_{\al \al}$ $\exists$ an index $i$ such that $(u_\al)_i > 0$ and $(A_{\al \al}u_\al)_i \geq 0.$ For such an index $i,$ $(A_{\al \al}u_\al + v)_i \geq 0.$ Hence $A \in E_0.$ We consider the following two cases:\\
	\NI Case I: Firstly we take $x = [x_{\alpha}, 0]^T,$ where $\alpha \in \{1, 2, \cdots, (n-1)\}.$ Then suppose $x^TAx = 0, \ x \geq 0,$ but in this case $Ax \ngeq 0.$\\
	\NI Case II: Take $x = [x_{\alpha}, x_{\bar{\alpha}}]^T,$ where $x_{\alpha}, x_{\bar{\alpha}} \geq 0.$ Then suppose for this $x,$ \ $x^TAx = 0,$ but $Ax \ngeq 0.$
So the vector $x$ for which $x^TAx = 0, \ Ax \geq 0, \ x \geq 0,$ are the zero vector and $[0, 0,\cdots, c]^T, c>0$ and for both cases $A^Tx \leq 0.$ \\
Hence $A$ is $E_{0}^{s}$ matrix. Now it is easy to show that for $x = [0, 0,\cdots, c]^T,$ $(A^Tx)_i \neq 0$ $\implies$ $(Ax)_i \neq 0$ for each $i.$ Hence $A \in \tilde{E^s_0}.$ 
\end{proof}
\vsp
\begin{remk} \label{connected}
Suppose $A \in R^{n\times n}$ with $A =$$\begin{bmatrix}
A_{\al\al} & + \\
- & 0 \\
\end{bmatrix}$ and $A_{\al\al} \in P_{0} \cap Q.$ Then $A$ is a  connected matrix $(E_{c})$ from the Theorem \ref{tp} of \cite{tpsriparna}. \end{remk}
\begin{remk}
	Suppose $A \in R^{n\times n}$ with $A =$	$\begin{bmatrix}
	A_{\al\al} & + \\
	- & 0 \\
	\end{bmatrix}$ and $A_{\al\al} \in P_{0} \cap Q.$ Now as $A \in E_{c}$ so $A \in E_{c} \cap Q_{0}$ and by the Theorem \ref{cao}, Lemke's algorithm processes LCP($q,A$).
\end{remk}

\bt
Suppose that $A \in R^{n\times n}$ with $A =$	$\begin{bmatrix}
A_{\al\al} & + \\
- & 0 \\
\end{bmatrix}$ and $A_{\al\al} \in P_{0} \cap Q.$ Then $A \in P_0.$
\et
\begin{proof}
Since $A \in R^{n\times n}$ with $A =$	$\begin{bmatrix}
A_{\al\al} & + \\
- & 0 \\
\end{bmatrix}$ and $A_{\al\al} \in P_{0} \cap Q$ then by the Remark \ref{connected}, $A \in E_c.$ Again by the Theorem \ref{fullysemi} $A \in E_{0}^{f}.$ As $A \in \tilde{E^{s}_0}$ by the Theorem \ref{structure}, $A \in L$ by the Theorem \ref{L_2}. By applying degree theory, $A \in P_0$ in view of Corollary 3.1 of \cite{mohan2001classes}.
\end{proof}

\begin{remk}
	Gowda and Jones \cite{jones-gowda} raised the following open problem: Is it true that $P_0 \cap Q_0 = E_c \cap Q_0?$ Cao and Ferris \cite{cao-ferris} showed that $P_0 \cap Q_0 = E_c \cap Q_0$ is true for second order matrices. We settle the above open problem partially by considering a subclass $P_0 \cap \tilde{E^s_0}$ of $P_0 \cap Q_0.$ 
\end{remk}

In general, SOL$(q, A)$ is not bounded for every $q \in int \ pos[-A, I]$ and $A \in \tilde{E^{s}_{0}}.$ Here we establish the following results.
\bt \label{bounded}
Let $A\in \tilde{E_{0}^s}$ and SOL$(q, A)$ is not bounded for all $q \in int \ pos[-A, I].$ Suppose $r \in K(A)$ and $\exists$ vectors $z$ and $z^{\lambda} = \hat{z} + \lambda z$ such that $z \in$ SOL$(0, A) \setminus \{0\},$ $z^{\lambda} \in$ SOL$(q, A)$ $\forall$ $\lambda \geq 0$ and $w \in$ SOL$(r, A).$ Then $(z^{\lambda} - w)_{\alpha}(A(z^{\lambda}-w))_{\alpha} < 0$ $\forall$ $\alpha = \{i : z_i \neq 0\}.$
\et
\begin{proof} Suppose $A \in \tilde{E_{0}^{s}}$ and  SOL$(q, A)$ is not bounded for all $q \in int \ pos[-A, I].$ Note that $A \in E_{0}^{s} \ \cap \ L_2$ as shown in Theorem \ref{L_2} and $q \in int \ pos[-A, I]$ and there exist vectors $z$ and $z^\lambda = \hat{z} + \lambda z$ such that $z \in$ SOL$(0, A) \setminus  \{0\}$ and $z^\lambda \in$ SOL$(q, A)$ $\forall \lambda \geq 0.$
We select an $r \in K(A)$ such a way that $\alpha = \{i : z_i \neq 0\}.$ Then
$r_i - q_i < 0$
Now for sufficiently large $\lambda,$ $(z^\lambda - w)_\alpha > 0$ and $w \in$ SOL$(r, A).$ We write
\begin{center}
	$(A(z^\lambda - w))_\alpha = -q_\alpha - (Aw)_\alpha \leq - q_\alpha + r_\alpha < 0.$
\end{center}
This implies
\begin{center}
	$(z^\lambda - w)_\alpha(A(z^\lambda - w))_\alpha < 0.$
\end{center}
\end{proof}
\NI However strict inequality does not hold in case of $\alpha \neq \{i : z_i \neq 0\}.$

\begin{theorem}
	Let $A\in \tilde{E_{0}^s}$ and SOL$(q, A)$ is not bounded for all $q \in int \ pos[-A, I].$ Suppose $r \in K(A)$ and $\exists$ vectors $z$ and $z^{\lambda} = \hat{z} + \lambda z$ such that $z \in$ SOL$(0, A) \setminus \{0\},$ $z^{\lambda} \in$ SOL$(q, A)$ $\forall$ $\lambda \geq 0$ and $\tilde{z} \in$ SOL$(r, A).$ Then $(z^{\lambda} - w)_{\alpha}(A(z^{\lambda}-w))_{\alpha} \leq 0$ $\forall$ $\alpha = \{i : \hat{z}_{i} \geq 0, z_{i} = 0\}.$
\end{theorem}
\begin{proof}
The first part of the proof follows from the proof of Theorem \ref{bounded}. Now
we select an $r \in K(A)$ and consider $\alpha = \{i : z_i \neq 0\}.$ We select an $r \in K(A)$ and consider $\alpha = \{i : z_i = 0\}.$ Then
$r_i - q_i \geq 0.$ Now for sufficiently large $\lambda,$ $(z^\lambda - w)_\alpha > 0$ and $w \in$ SOL$(r, A).$ Now we consider following two cases.\\
Case I:
Let $\alpha = \{i : \hat{z}_i > 0, z_i = 0\}.$ Then $r_i - q_i = 0.$
We write
\begin{center}
	$ \begin{array}{ll}
	(z^\lambda - w)_i(A(z^\lambda - w))_i&=(z^\lambda - w)_i((Az^\lambda)_i-(Aw)_i + q_i - r_i)\\
	&=z^\lambda_{i}((Az^\lambda)_i + q_i) - w_i((Az^\lambda)_i + q_i) +\\ & z^\lambda_{i}(-(Aw)_i - r_i) - w_i(-(Aw)_i - r_i)\\
	&\leq 0.
	\end{array} $
\end{center}
Case II:
Let $\alpha = \{i : z_i = \hat{z}_i = 0\}.$ Then $r_i - q_i > 0.$
We write
\begin{center}
	$ \begin{array}{ll}
	(z^\lambda - w)_i(A(z^\lambda - w))_i&=-w_i((Az^\lambda)_i - (Aw)_i)\\
	&\leq -w_i((Az^\lambda)_i - (Aw)_i + q_i - r_i)\\
	&=-w_i(Az^\lambda + q)_{i} + w_i(Aw + r)_i\\
	&=-w_i(Az^\lambda + q)_{i}\\
	&\leq 0.
	\end{array} $
\end{center}
\end{proof}

Now we show the condition for which SOL$(q, A)$ is compact where $A \in \tilde{E_{0}^{s}}.$ To establish the result we use game theoretic approach and Ville's theorem of alternative.
\bt \label{R_0thm}
Suppose $A\in \tilde{E_{0}^{s}}$ with $v(A) > 0.$ Then SOL$(q, A)$ is compact.
\et
\begin{proof}
	By theorem \ref{L_2}, $\tilde{E_{0}^{s}} \subseteq E_0 \cap L_2.$ Since $v(A) > 0,$ So $A \in E_{0}^{s} \cap Q.$ Now to establish $A\in R_0$ it is enough to show that LCP($0,A$) has only trivial solution. Suppose not, then LCP($0,A$) has nontrivial solution, i.e. say, $0 \neq x \in$ SOL$(0, A)$ then $0 \neq x \geq 0,$ $Ax \geq 0$ and $x^TAx = 0.$ Since $A\in E_{0}^{s},$ we can write $A^Tx \leq 0.$ Now $A^Tx \leq 0,\;0\neq x\geq 0 $ has a solution. According to Ville's theorem of alternative, there does not exist $x > 0$ such that $Ax > 0.$ However, $Ax>0,\;x>0$ has a solution since $A \in Q.$ This is a contradiction. Hence LCP$(0,A)$ has only trivial solution. Therefore $A\in Q \cap R_0.$ Now by the Theorem \ref{lopez}, $A \in Q_b.$ Hence SOL$(q, A)$ is nonempty and compact. 
\end{proof}

We illustrate the result with the help of an example.
\begin{examp}
	Consider the matrix
	$ \dsp {
		A = \left[ \begin{array}{rrr}
		0 & 2 & 1\\
		1 & 0 & 1\\
		-2 & -2 & 1
		\end{array} \right]
	} .$
	Now $x^TAx = 3x_1x_2 + x_3^{2} - x_3(x_1 + x_2).$ Now we consider the following four cases.\\
		 \NI Case I: For $x_1 = 0, \ x_2 = k, \ x_3 = 0,$ where $k > 0.$ Here $x \geq 0, \ x^TAx = 0$ holds but in this case $Ax \ngeq 0.$\\
		 \NI Case II: For $x_1 = k, \ x_2 = 0, \ x_3 = 0,$ where $k > 0.$ Here $x \geq 0, \ x^TAx = 0$ holds but in this case $Ax \ngeq 0.$\\
		 \NI Case III: $x_1 = 0, \ x_2 = k, \ x_3 = k,$ where $k > 0.$ Here $x \geq 0, \ x^TAx = 0$ holds but in this case $Ax \ngeq 0.$\\
		 \NI Case IV: $x_1 = k, \ x_2 = 0, \ x_3 = k,$ where $k > 0.$ Here $x \geq 0, \ x^TAx = 0$ holds but in this case $Ax \ngeq 0.$
	
	Hence zero vector is the only vector for which $x \geq 0, \ Ax \geq 0, \ x^TAx = 0$ implies $A^Tx \leq 0$ holds. So $A \in E_{0}^{s}$-matrix. Also it is clear that $A \in \tilde{E_0^s}.$ Here we get that LCP$(0, A)$ has unique solution. Hence $A \in R_0.$ 
\end{examp}

We state some notion of stability of a linear complementarity problem at a solution point.
\begin{defn}
A solution $x^{*}$ is said to be stable if there are neighborhoods $V$ of $x^{*}$ and $U$ of $(q, A)$ such that\\
(i) for all $(\bar{q}, \bar{A}) \in U,$ the set SOL$(\bar{q}, \bar{A}) \cap V \neq \emptyset.$\\
(ii) sup$\{\|y - x^{*}\|: y \in$ SOL$(\bar{q}, \bar{A}) \cap V \neq \emptyset \}$ goes to zero as $(\bar{q}, \bar{A})$ approaches $(q, A).$
\end{defn}
\begin{defn}
A solution $x^{*}$ is said to be strongly stable if there exists a neighborhood $V$ of $x^{*}$ such that the set SOL$(\bar{q}, \bar{A}) \cap V$ is singleton.
\end{defn}
\begin{defn}
	A solution $x^{*}$ is said to be locally unique if there exists a neighborhood $V$ of $x^{*}$ such that SOL$(\bar{q}, \bar{A}) \cap V = \{x^{*}\}.$
\end{defn}

The following result shows that the solution set of LCP($q, A$) is stable when $A \in \tilde{E^s_0}.$
\bt
Suppose $A \in \tilde{E_{0}^{s}}$ with $v(A) > 0,$ if the LCP$(q, A)$ has unique solution $x^{*},$ then LCP$(q, A)$ is stable at $x^{*}.$
\et
\begin{proof}
	As $A \in \tilde{E_{0}^{s}}$ with $v(A) > 0,$ then by the Theorem \ref{R_0thm}, $A \in R_0.$ Again as shown in the Theorem \ref{pangresult}, $A \in$ int$(Q) \cap R_0.$ So by the Theorem \ref{pangtheorem}, if the LCP$(q, A)$ has unique solution $x^{*},$ then LCP$(q, A)$ is stable at $x^{*}.$
\end{proof}

\section{Iterative algorithm to process LCP$(q, A)$}
Aganagic and Cottle \cite{aganagic} proved that Lemke's algorithm processes LCP$(q, A)$ if $A \in P_0 \cap Q_0.$ Todd and Ye \cite{todd} proposed a projective algorithm to solve linear programming problem considering a suitable merit function. Using the same merit function Pang \cite{Pang} proposed an iterative descent type algorithm with a fixed value of the parameter $\kappa$ to process LCP$(q, A)$ where $A$ is a row sufficient matrix. Kojima et al. \cite{kojima1991unified} proposed an interior point method to process $P_0$-matrices using similar type of merit function. Here we propose a modified version of interior point algorithm by using a dynamic $\kappa$ for each iterations in line with Pang \cite{Pang} for finding solution of LCP$(q, A)$ given that $A \in \tilde{E^{s}_{0}}.$ Note that $\tilde{E^{s}_{0}}$ contains $P_0$-matrices as well as non $P_0$-matrices. We prove that the search directions generated by the algorithm are descent and show that the proposed algorithm converges to the solution under some defined conditions. 

\

\NI{\bf Algorithm. }\\
Let $z>0$, $w=q+Az>0,$ and $\psi:R^n_{++}\times R^n_{++}\rightarrow R$ such that $\psi(z,w)=\kappa^k\,\mbox{log}(z^Tw)-\sum_{i=1}^{n}\mbox{ log }(z_iw_i)\geq 0.$ Further suppose $\rho^k =$ $\min_{i}$$\{z_i^kw_i^k\}$ and $\kappa^k >$ $\max$$(n, \frac{z^Tw}{\rho^k})$ for $k$-th iteration. 
\begin{description}
	\item[Step 1:] Let $\beta \in (0,1)$ and $\sigma \in (0,\frac{1}{2})$ following line search step and $z^0$ be a strictly feasible point of LCP$(q, A)$ and  $w^0=q+Az^0>0.$
	\begin{center}
		$\nabla_{z}\psi_{k}= \nabla_{z}\psi(z^k,w^k)$,~~~~~~$\nabla_{w}\psi_{k}= \nabla_{w}\psi(z^k,w^k)$
	\end{center}
	and
	\begin{center}
		$Z^k=diag(z^k)$,~~~~~~$W^k=diag(w^k)$.
	\end{center}
	\item[Step 2:] Now to find the search direction, consider the following problem
	\begin{center}
		minimize~~~~~~$(\nabla_{z}\psi_{k})^{T}d_{z} + (\nabla_{w}\psi_{k})^{T}d_{w}$\\
	\end{center}
	\begin{center}
		subject to~~~~ $d_{w}=Ad_{z}$,~~~~~~$\|(Z^k)^{-1}d_{z}\|^{2}+\|(W^k)^{-1}d_{w}\|^{2}\leq \beta^{2}.$
	\end{center}
	\item[Step 3:] Find the smallest $m_{k} \geq 0$ such that
	\begin{center}
		$\psi(z^k+2^{-m_k}d_{z}^{k},w^k+2^{-m_k}d_{w}^{k})-\psi(z^{k},w^{k})\leq \sigma2^{-m_k}[(\nabla_{z}\psi_{k})^{T}d_{z}^{k} + (\nabla_{w}\psi_{k})^{T}d_{w}^{k}].$
	\end{center}
	\item[Step 4:] Set
	\begin{center}
		$(z^{k+1},w^{k+1})=(z^{k},w^{k})+ 2^{-m_{k}}(d_{z}^{k},d_{w}^{k}).$
	\end{center}
	\item[Step 5:]If $(z^{k+1})^Tw^{k+1} \leq \epsilon,$ where $\epsilon$ is a very small positive quantity, stop else $k = k+1.$
\end{description}

\begin{remk}
	The algorithm is based on the existence of a strictly feasible point. As $A \in \tilde{E^{s}_{0}}$ implies $A \in Q_0$ in view of Theorem \ref{L_2} then existence of a strictly feasible points for such a matrix will eventually lead to the solution of LCP$(q, A).$
\end{remk}

Now we prove the following lemma for $E_0$-matrices.
\bl\label{lemma}
Suppose $A\in E_0,$ $z>0,$ $w=q+Az>0,$ and $\psi:R^n_{++}\times R^n_{++}\rightarrow R$ such that $\psi(z, w)=\kappa\,\mbox{log}(z^Tw)-\sum_{i=1}^{n}\mbox{ log }(z_iw_i).$ Further suppose $\rho^k =$ $\min$$_i$$\{z_i^kw_i^k\}$ and $\kappa^k >$ $\max$$(n, \frac{z^Tw}{\rho^k})$ for each $k$th iteration. Then the search direction $(d^k_z, d^k_w)$ generated by the algorithm is descent direction.

\el
\begin{proof} Let us consider $r^k=\nabla_{z}\psi_{k}+A^{T}\nabla_{w}\psi_{k}$ and first we show that $r^k \neq 0$ for $k$th iteration. Consider the merit function $z>0$, $w=q+Az>0$ and $\psi:R^n_{++}\times R^n_{++}\rightarrow R$ such that $\psi(z,w)=\kappa\,\mbox{log}(z^Tw)-\sum_{i=1}^{n}\mbox{ log }(z_iw_i)\geq 0.$ Note that
	$$ \begin{array}{ll}
	\big(\nabla_z \psi(z,w)\big)_i&=\frac{\kappa}{z^Tw}v_i-\frac{1}{z_iw_i}w_i\\
	&=w_i\big[\frac{\kappa}{z^Tw}-\frac{1}{z_iw_i}\big].
	\end{array}$$
	Similarly we show
	$$ \begin{array}{ll}
	\big(\nabla_w \psi(z, w)\big)_i&=z_i\big[\frac{\kappa}{z^Tw}-\frac{1}{z_iw_i}\big].
	\end{array}$$
	Again for $k$th iteration $\kappa^k >$ $\max$$(n, \frac{z^Tw}{\rho^k})$ where $\rho^k =$ $\min$$_i$$\{z_i^kw_i^k\}.$ This implies
	\begin{center}
		$z_i(\frac{\kappa^k}{z^Tw}-\frac{1}{z_iw_i}) > 0.$
	\end{center}
	Therefore $\big(\nabla_w \psi(z, w)\big)_i > 0$ $\forall i.$ In a similar way we can show that $\big(\nabla_z \psi(z, w)\big)_i >0$ $\forall i.$ Now $A \in E_0.$ So $A^T \in E_0.$ By the definition of semimonotonicity for $\big(\nabla_w \psi(z, w)\big) > 0$  $\exists$ a $j$ such that $(A^T\nabla_w\psi (z, w))_j \geq 0.$ Therefore $(\nabla_z\psi (z, w))_j+(A^T\nabla_v\psi (z, w))_j \neq 0$ for atleast one $j.$ Hence $\nabla_z\psi (z, w)+A^T\nabla_v\psi (z, w)\neq 0.$ We have $d_{z}^{k}=-\frac{(A^{k})^{-1}r^{k}}{\tau_{k}}$, $d_{w}^{k}=Ad_{z}^{k}$ from the algorithm. Again $A^{k}= (Z^{k})^{-2}+A^{T}(W^{k})^{-2}A$ is positive definite as
$$ \begin{array}{ll}
x^TA^T(W)^{-2}Ax &=(Ax)^T(W)^{-2}Ax\\
&=(y)^T(W)^{-2}y
\end{array} $$
and $(y)^T(W)^{-2}y\geq 0,\;\forall\;y\in R^{n},$ $A^T(W)^{-2}A$ is positive semidefinite. So $\tau_{k}= \dfrac{\sqrt{(r^{k})^{T}(A^{k})^{-1}r^{k}}}{\beta}$ is positive. Now we show that $(\nabla_z\psi_k)^Td_z^k + (\nabla_w\psi_k)^Td_w^k<0.$ We derive
$$ \begin{array}{ll}
(\nabla_z\psi_k)^Td_z^k + (\nabla_w\psi_k)^Td_w^k&=\big[\nabla_z\psi_k + A^T\nabla_w\psi_k\big]^Td^k_w\\
&=-\frac{1}{\tau_k}(\sqrt{(r^k)^T(A^k)^{-1}r^k})^2\\
&=-\tau_k\beta^2<0.
\end{array} $$
\NI We consider $\psi(z^k+2^{-m_k}d_{z}^{k},w^k+2^{-m_k}d_{w}^{k})-\psi(z^{k}, w^{k})\leq \sigma2^{-m_k}[(\nabla_{z}\psi_{k})^{T}d_{z}^{k} + (\nabla_{w}\psi_{k})^{T}d_{w}^{k}].$
Since $0<\beta, \sigma<1,$ we say $\psi(z^k+2^{-m_k}d_{z}^{k}, w^k+2^{-m_k}d_{w}^{k})-\psi(z^{k}, w^{k})< 0.$
Hence $(d^k_z, d^k_w)$ is descent direction in this algorithm.
\end{proof}

\begin{remk}
Note that the Lemma \ref{lemma} is true for $\tilde{E^{s}_{0}}$-matrices as $\tilde{E^{s}_{0}} \subseteq E_0.$
\end{remk}

We prove the following theorem to show that the proposed algorithm converges to the solution under some defined condition.
\vsp
\bt \label{thr6}
If $A\in \tilde{E^{s}_{0}}$ and LCP$(q, A)$ has a strictly feasible solution, then every accumulation point of $\{z^k\}$ is the solution of LCP$(q, A)$ i.e. algorithm converges to the solution.
\et
\begin{proof} If there exists strictly feasible points then LCP$(q, A)$ has a solution where $A \in \tilde{E^{s}_{0}}.$ Let us consider the subsequences $\{z^k : k \in \omega\}.$ Suppose $\tilde{z}$ is the limit of the subsequence and $\tilde{w} = q + A\tilde{z}.$ Again we know $\psi(\tilde{z}, \tilde{w}) < \infty.$ So either $\tilde{z}^{T}\tilde{w} = 0$ or $(\tilde{z}, \tilde{w}) > 0.$ If the first case happen, then $(\tilde{z}, \tilde{w})$ is a solution. So let us consider that $(\tilde{z}, \tilde{w}) > 0.$ Also suppose $\tilde{r}$ and $\tilde{A}$ are the limits of the subsequences $\{r^{k} : k \in \omega\}$ and $\{A^{k} : k \in \omega\}$ respectively. Consider $\tau^{k}$ converges to $\tilde{\tau} = \dfrac{\sqrt{\tilde{r}^{T}\tilde{A}^{-1}\tilde{r}}}{\beta}> 0,$ where $\tilde{A}$ remains positive definite. $(\tilde{d_{z}}, \tilde{d_{w}})$ be the limits of the sequence of direction $(d_{z}^{k}, d_{w}^{k}).$ So from the algorithm we get
\begin{center}
	$\tilde{d_{z}}=-\frac{\tilde{A}^{-1}\tilde{r}}{\tilde{\tau}}$,~~~~~~${\tilde{d_{w}}=A\tilde{d_{z}}}.$
\end{center}
Now as $\{\psi(z^{k+1}, w^{k+1}) - \psi(z^{k}, w^{k})\}$ converges to zero and since $\lim m_{k} = \infty$ as $k \rightarrow \infty,$ $\{(z^{k+1}, w^{k+1}) : k \in \omega\}$ and $\{(z^{k} + 2^{-(m_{k} - 1)}d_{z}^{k}, w^{k} + 2^{-(m_{k} - 1)}d_{w}^{k}) : k \in \omega\}$ converges to $(\tilde{z}, \tilde{w})$. As $m_{k}$ is the smallest non-negative integers, we have,
\begin{center}
	$\frac{\psi(z^{k} + 2^{-(m_{k} - 1)}d_{z}^{k}, w^{k} + 2^{-(m_{k} - 1)}d_{w}^{k}) - \psi(z^{k}, w^{k})}{2^{-(m_{k} - 1)}} > -\sigma \beta^{2}\tau_{k}.$
\end{center}
Again on the other hand from the algorithm,
\begin{center}
	$\frac{\psi(z^{k+1}, w^{k+1}) - \psi(z^{k}, w^{k})} {2^{-m_{k}}} \leq -\sigma \beta^{2}\tau_{k}.$
\end{center}
Now taking limit $k \rightarrow \infty,$ we write,
\begin{center}
	$\nabla_z\psi(\tilde{z}, \tilde{w})^T\tilde{d_z} + \nabla_w\psi(\tilde{z}, \tilde{w})^T\tilde{d_w} = -\sigma\tilde{\tau}\beta^{2}.$
\end{center}
Again from Lemma \ref{lemma} we know,
\begin{center}
	$(\nabla_z\psi_k)^Td_z^k + (\nabla_w\psi_k)^Td_w^k = -\tau_k\beta^2.$
\end{center}
Hence by taking limit $k \rightarrow \infty,$ we get
\begin{center}
	$\nabla_z\psi(\tilde{z}, \tilde{w})^T\tilde{d_z} + \nabla_w\psi(\tilde{z}, \tilde{w})^T\tilde{d_w} = -\tilde{\tau}\beta^{2}.$
\end{center}
Therefore we arrive at a contradiction. So our proposed algorithm converges to the solution.
\end{proof} 

\section{Numerical illustration}
A numerical example is considered to demonstrate the effectiveness and efficiency of the proposed algorithm. 
\begin{examp}

We consider the following example of LCP$(q, A),$
where
\begin{center}
	$A$ =  $\left(\begin{array}{rrr}
	0 & 1 & 1  \\
	2 & 0 & 2 \\
	-2 & -5 & 0
	\end{array}\right)$ and
	\ $q$ =  $\left(\begin{array}{r}
	-4\\
	-7\\
	 10
	\end{array}\right).$
\end{center}
It is easy to show that $A \in \tilde{E_{0}^{s}}.$ We apply proposed algorithm to find solution of the given problem. According to Theorem \ref{thr6} algorithm converges to solution with $z^0, w^0 > 0.$ To start with we initialize $\beta = 0.5,$ $\gamma = 0.5,$ $\sigma = 0.2,$ and $\epsilon = 0.00001.$ We set $z^0$ =   $\left(\begin{array}{rrr}
 1\\
 1 \\
 5
\end{array} \right)$
and obtain $w^0$ =   $\left(\begin{array}{r}
2 \\
5 \\
3
\end{array} \right).$
	\begin{table}[h!]
		\centering
		\resizebox{\columnwidth}{!}{%
		\begin{tabular}{|c|c|c|c|c|c|}
			\hline
			Iteration (k) & $z^k$ & $w^k$ &  $d_z^k$ & $d_w^k$ & $\psi(z^k, w^k)$\\
			\hline

			1 & $\left( \begin{array}{rrr}  1.05 \\ 1.09 \\ 4.76 \end{array} \right)$ & $\left( \begin{array}{rrr} 1.85 \\ 4.62 \\ 2.42 \end{array} \right)$& $\left( \begin{array}{rrr} 0.106 \\ 0.189 \\ -0.487 \end{array} \right)$ & $\left( \begin{array}{rrr}  -0.298 \\ -0.761 \\ -1.155 \end{array} \right)$ & 29.3308\\
			\hline
			
			\hline
			2 & $\left( \begin{array}{rrr}  1.1 \\ 1.17 \\ 4.53 \end{array} \right)$ & $\left( \begin{array}{rrr} 1.7 \\ 4.25 \\ 1.94 \end{array} \right)$& $\left( \begin{array}{rrr} 0.0853 \\ 0.1607 \\ -0.4551 \end{array} \right)$ & $\left( \begin{array}{rrr} -0.294 \\ -0.74 \\ -0.974 \end{array} \right)$ & 23.2919\\
			
			\hline
			
			

			 \hline
			 
			 \vdots  & \vdots & \vdots & \vdots & \vdots & \vdots\\
			 \hline


			 50 & $\left( \begin{array}{rrr}  1.07 \\ 1.57 \\ 2.43 \end{array} \right)$ & $\left( \begin{array}{rrr} 0.00608 \\ 0.00389 \\ 0.00281 \end{array} \right)$& $\left( \begin{array}{rrr} 0.00047 \\ -0.00017 \\ -0.00154 \end{array} \right)$ & $\left( \begin{array}{rrr} -0.00171 \\ -0.00215 \\ -0.00009 \end{array} \right)$ & 2.4617\\
			 \hline
			 
			 \vdots  & \vdots & \vdots & \vdots & \vdots & \vdots\\
			 \hline
			 
			 96 & $\left( \begin{array}{rrr}  1.07 \\ 1.57 \\ 2.43  \end{array} \right)$ & $\left( \begin{array}{rrr} 0.00001 \\ 0.000000 \\ 0.00000 \end{array} \right)$& $\left( \begin{array}{rrr} -0.000001 \\ -0.00000 \\ -0.000003 \end{array} \right)$ & $\left( \begin{array}{rrr} -0.00000 \\ -0.00000 \\ -0.00000 \end{array} \right)$ & 1.1684\\
			 \hline


			 97 & $\left( \begin{array}{rrr}  1.07 \\ 1.57 \\ 2.43 \end{array} \right)$ & $\left( \begin{array}{rrr} 0.00001 \\ 0.000009 \\ 0.000005 \end{array} \right)$& $\left( \begin{array}{rrr} 0.000002 \\ 0.000000 \\ -0.000000 \end{array} \right)$ & $\left( \begin{array}{rrr} -0.000000 \\ -0.000000 \\ -0.000000 \end{array} \right)$ & 1.1684\\
			 \hline

			 

			 \vdots  & \vdots & \vdots & \vdots & \vdots & \vdots\\
			 \hline

			 100 & $\left( \begin{array}{rrr}  1.07 \\ 1.57 \\ 2.43 \end{array} \right)$ & $\left( \begin{array}{rrr} 0.00000 \\ 0.00000 \\ 0.00000 \end{array} \right)$& $\left( \begin{array}{rrr} 0.000000 \\ -0.000000 \\ -0.000000 \end{array} \right)$ & $\left( \begin{array}{rrr} -0.000000 \\ -0.000000 \\ 0.00000 \end{array} \right)$ & 1.0565\\
			 \hline

		\end{tabular}%
		}
		\caption{Summary of computation for the proposed algorithm} \label{t1}
	\end{table}

\pagebreak	
Table \ref{t1} summarizes the computations for the first 2 iterations, 50th iteration and 96th, 97th iteration and 100th iteration. At the 100th iteration, sequence $\{z^k\}$ and $\{w^k\}$ produced by the proposed algorithm converges to the solution of the given LCP$(q, A)$ i.e. $z^{*}$ =  $\left(\begin{array}{rrr}
1.0714 \\
1.5714 \\
2.4285
\end{array}\right)$ and $w^{*}$ =    $\left(\begin{array}{rrr}
0 \\
0 \\
0
\end{array}\right).$	
\end{examp}
\section{Concluding remark}
 In this article, we show that LCP$(q,A)$ is processable by Lemke's algorithm and the solution set of LCP$(q,A)$ is bounded if $A\in \tilde{E_0^s} \cap P_0,$ a subclass of $E_0^s \cap P_0.$ It can be shown that non-negative matrices with zero diagonal with atleast one $a_{ij} > 0$ with $i \neq j$ is not a $\tilde{E_{0}^{s}}$-matrix. Whether a matrix class belongs to $P_{0}\cap Q_{0}$ or not is difficult to verify. We find some conditions under which $\tilde{E_{0}^{s}}$-matrix will belong $P_{0}\cap Q_{0}$ which will motivate further study and applications in matrix theory. Finally we propose an iterative and descent type interior point method to compute solution of LCP$(q, A).$


\bibliographystyle{plain}
\bibliography{bibfile}

\end{document}